\documentclass[12pt]{amsart}

\newtheorem{theorem}{Theorem}[section]

\newtheorem{lemma}[theorem]{Lemma}
\newtheorem{remark}[theorem]{Remark}

\numberwithin{equation}{section}

\begin{document}

\baselineskip=16pt

\title[Holomorphic Cartan geometries]{Holomorphic Cartan
geometries, Calabi--Yau manifolds and rational curves}

\author[I. Biswas]{Indranil Biswas}

\address{School of Mathematics, Tata Institute of Fundamental
Research, Homi Bhabha Road, Bombay 400005, India}

\email{indranil@math.tifr.res.in}

\author[B. McKay]{Benjamin McKay}

\address{School of Mathematical Sciences,
University College Cork, National University of Ireland}

\email{B.McKay@ucc.ie}



\date{}

\begin{abstract}
We prove that
if a Calabi--Yau manifold $M$ admits a holomorphic Cartan
geometry, then $M$ is covered by a complex torus. This is done by
establishing the Bogomolov inequality for semistable sheaves on
compact K\"ahler manifolds. We also classify all holomorphic
Cartan geometries on rationally connected complex projective
manifolds.
\end{abstract}

\maketitle

\section{Introduction} 

In \cite{Mc3}, the second author conjectured that
the only Calabi--Yau manifolds that bear holomorphic Cartan
geometries are those covered by the complex tori.
Our first theorem confirms this conjecture. More precisely,
we prove:

\begin{theorem}\label{thi1}
Let $M$ be a compact connected K\"ahler manifold with
$c_1(M)\,=\, 0$. If $M$ is equipped with a holomorphic
Cartan geometry,
then there is an \'etale Galois holomorphic covering map
$$
A\, \longrightarrow\, M\, ,
$$
where $A$ is a complex torus.
\end{theorem}

This generalizes our earlier result in \cite{BM}; in \cite{BM},
Theorem \ref{thi1} was proved under the assumption that $M$
is a complex projective manifold.

We then consider Cartan geometries on rationally connected
projective varieties.

Let $G$ be a complex Lie group and $H\,\subset\, G$
a closed complex subgroup. The following theorem is proved in
Section \ref{sec4}.

\begin{theorem}\label{thi2}
Let $M$ be a rationally connected complex projective
manifold and $(E_H\, ,\theta)$ a holomorphic Cartan geometry
of type $G/H$ on $M$. Then
the homogeneous space $G/H$ is a rational projective variety, and
$M\, \cong\, G/H$. Furthermore,
$(M\, ,E_H\, ,\theta)$ is holomorphically isomorphic to
the tautological Cartan geometry of type $G/H$ on $G/H$.
\end{theorem}

While we were finalizing this work, a preprint \cite{Du}
appeared where Theorem \ref{thi1} is proved under the assumption
that the image of $H$ in $\text{Aut}(\text{Lie}(G))$ for the
adjoint representation is algebraic (see \cite[p. 1]{Du}).
Theorem \ref{thi2} is known for parabolic geometries
\cite{Mc2}, \cite{Bi3}.

This material is based upon works supported by the Science Foundation
Ireland under Grant No. MATF634.

\section{Bogomolov inequality}

Let $M$ be a compact connected K\"ahler manifold
equipped with a K\"ahler form $\omega$.

For a torsionfree
coherent analytic sheaf $V\, \not=\, 0$ on $M$, define
$$
\text{degree}(V)\, :=\, (c_1(V)\wedge \omega^{d-1})\cap [M]
\,\in\, {\mathbb R}\, ,
$$
where $d\, =\, \dim_{\mathbb C} M$. We recall that $V$
is called \textit{semistable} (respectively, \textit{stable}) if
for all torsionfree coherent analytic sheaves $V'\, \subset\, V$
with $\text{rank}(V')\, <\, \text{rank}(V)$, the inequality
$$
\frac{\text{degree}(V')}{\text{rank}(V')}\, \leq\,
\frac{\text{degree}(V)}{\text{rank}(V)} ~\,~\,~\,\text{(respectively,
~}~\,\frac{\text{degree}(V')}{\text{rank}(V')}\, <\,
\frac{\text{degree}(V)}{\text{rank}(V)}\text{)}
$$
holds. A semistable sheaf $V$ is called \textit{polystable} if
it is a direct sum of stable sheaves.

\begin{lemma}\label{lem1}
Let $V$ be a semistable sheaf on $M$. Then
$$
((2r\cdot c_2(V) - (r-1)c_1(V)^2)\cup \omega^{d-2})\cap [M]
\,\geq\, 0\, ,
$$
where $r$ is the rank of $V$.
\end{lemma}

\begin{proof}
If $V$ is polystable and reflexive, then this is proved
in \cite{BS} (see \cite[p. 40, Corollary 3]{BS}).

First assume that $V$ is polystable.

The double dual $V^{**}$ is reflexive. Since
$V$ is torsionfree, the natural homomorphism
$V\, \longrightarrow\, V^{**}$ is injective. Hence
we have a short exact sequence of coherent
analytic sheaves
\begin{equation}\label{e2}
0\,\longrightarrow\, V\,\longrightarrow\, V^{**}
\,\longrightarrow\, T \,\longrightarrow\, 0
\end{equation}
on $M$,
where $T$ is a torsion sheaf, and the (complex) codimension
of the support of $T$ is at least two (see
\cite[p. 154, Proposition (4.14)]{Kob}).

Since $V$ is polystable, we know that $V^{**}$ is also
polystable (follows from the fact that the double
dual of a stable sheaf is stable). Hence
\begin{equation}\label{e1}
((2r\cdot c_2(V^{**}) - (r-1)c_1(V^{**})^2)\cup \omega^{d-2})
\cap [M]\,\geq\, 0
\end{equation}
\cite[p. 40, Corollary 3]{BS}.

Consider the quotient $T$ in \eqref{e2}.
Let $S_1\, ,\cdots\, , S_\ell$ be the components of the
support of $T$ of codimension two. Let $r_i$ be the
rank of $T\vert_{S_i}$. Using \eqref{e2} we conclude
that
\begin{equation}\label{e3}
c_1(V)\, =\, c_1(V^{**})~\,~\, \text{~and~}~\,~\,
c_2(V)\, =\, c_2(V^{**})+\sum_{i=1}^\ell r_i\cdot
\psi([S_i])\, ,
\end{equation}
where $\psi\, :\, H_{2d-4}(M,\, {\mathbb Q})\,\longrightarrow\,
H^4(M,\, {\mathbb Q})$ is the isomorphism given by
the Poincar\'e duality.

Since $\text{Volume}_\omega(S_i) \,=\, (\omega^{d-2}\cup \psi([S_i]))
\cap [M]\, \geq\, 0$, from \eqref{e1} and
\eqref{e3} we conclude that
\begin{equation}\label{e4}
((2r\cdot c_2(V) - (r-1)c_1(V)^2)\cup \omega^{d-2})\cap [M]
\,\geq\, 0\, .
\end{equation}

Finally, let $V$ be any semistable sheaf. There is a
filtration of coherent analytic subsheaves of $V$
$$
0\, =:\, V_0\, \subset\, V_1 \, \subset\, \cdots
\, \subset\, V_{n-1} \, \subset\, V_n\,=\, V
$$
such that each successive quotient $V_i/V_{i-1}$,
$i\, \in\, [1\, , n]$, is torsionfree and stable, and
furthermore,
$$
\frac{\text{degree}(V_i/V_{i-1})}{\text{rank}(V_i/V_{i-1})}
\,=\, \frac{\text{degree}(V)}{\text{rank}(V)}
$$
(see \cite[p. 175, Theorem (7.18)]{Kob}). Therefore,
$$
\widehat{V}\, :=\, \bigoplus_{i=1}^n V_i/V_{i-1}
$$
is a polystable sheaf. Since
$$
c_i(V)\, =\, c_i(\widehat{V})
$$
for all $i\, \geq\, 0$, from \eqref{e4} we conclude that
the inequality in the lemma holds. This completes the
proof of the lemma.
\end{proof}

\section{Cartan geometries and Calabi--Yau manifolds}\label{sec3}

Let $G$ be a complex Lie group and $H\,\subset\, G$
a closed complex subgroup. The Lie algebra of $G$
(respectively, $H$) will be denoted by $\mathfrak g$
(respectively, $\mathfrak h$). As before, $M$ is a compact
connected K\"ahler manifold.

A \textit{holomorphic Cartan geometry} on $M$ of type $G/H$
is a holomorphic principal $H$--bundle
\begin{equation}\label{e0}
E_H\, \longrightarrow\, M
\end{equation}
together with a holomorphic one--form on
$E_H$ with values in $\mathfrak g$
\begin{equation}\label{e00}
\theta\, \in\, H^0(E_H,\, \Omega^1_{E_H}\otimes_{\mathbb C}
{\mathfrak g})
\end{equation}
satisfying the following three conditions:
\begin{itemize}
\item $\tau^*_h\theta\, =\, \text{Ad}(h^{-1})\circ\theta$
for all $h\, \in\,H$, where $\tau_h\, :\, E_H\, \longrightarrow\,
E_H$ is the translation by $h$,

\item $\theta(z)(\zeta_v(z)) \, =\, v$ for all
$v\, \in\, {\mathfrak h}$ and $z\, \in\, E_H$, where $\zeta_v$
is the vector field on $E_H$ defined by $v$ using the
action of $H$ on $E_H$, and

\item for each point $z\, \in\, E_H$, the homomorphism from the
holomorphic tangent space
\begin{equation}\label{e-1}
\theta(z) \,:\, T_zE_H\, \longrightarrow\, {\mathfrak g}
\end{equation}
is an isomorphism of vector spaces.
\end{itemize}
(See \cite{S}.)

Let $(E_H\, ,\theta)$ be a holomorphic Cartan geometry of type $G/H$
on $M$. Let
\begin{equation}\label{eg}
E_G\, :=\, (E_H\times G)/H\, \longrightarrow\, M
\end{equation}
be the holomorphic principal $G$--bundle obtained by extending the
structure group of $E_H$ using the inclusion of $H$ in $G$ (the
action of $h\,\in\, H$ sends any $(z\, ,g)\, \in\,
E_H\times G$ to $(zh\, ,h^{-1}g)$). Let
$$
\theta_{\text{MC}}\, :\,
TG\, \longrightarrow\, G\times\mathfrak g
$$
be the $\mathfrak g$--valued Maurer--Cartan
one--form on $G$ constructed using
the left invariant vector fields.
Consider the $\mathfrak g$--valued holomorphic
one--form
$$
\widetilde{\theta}\, :=\, p^*_1 \theta + p^*_2\theta_{\text{MC}}
$$
on $E_H\times G$, where $p_1$ (respectively,
$p_2$) is the projection of $E_H\times G$ to
$E_H$ (respectively, $G$). This form $\widetilde{\theta}$
descends to a $\mathfrak g$--valued holomorphic one--form on the
quotient space $E_G$ in \eqref{eg}. This descended one--form
defines a holomorphic connection on $E_G$.

Therefore, the principal $G$--bundle $E_G$ in \eqref{eg} is
equipped with a holomorphic connection.

Let $\text{ad}(E_G)$ (respectively, $\text{ad}(E_H)$) be the adjoint
bundle of $E_G$ (respectively, $E_H$). So $\text{ad}(E_G)$
(respectively, $\text{ad}(E_H)$) is associated to $E_G$ for the
adjoint action of $G$ (respectively, $H$) on
$\mathfrak g$ (respectively, $\mathfrak h$).
There is a natural inclusion
$$
\text{ad}(E_H)\,\hookrightarrow\, \text{ad}(E_G)\, .
$$
Using the form $\theta$, the quotient $\text{ad}(E_G)/\text{ad}(E_H)$
gets identified with the holomorphic tangent bundle $TM$. Therefore,
we get a short exact sequence of holomorphic vector bundle on $M$
\begin{equation}\label{eq1}
0\,\longrightarrow\, \text{ad}(E_H)\,\longrightarrow\,\text{ad}(E_G)
\,\longrightarrow\, TM \,\longrightarrow\, 0\, .
\end{equation}

\begin{theorem}\label{thm1}
Assume that
\begin{itemize}
\item $c_1(M)\,=\, 0$, and

\item $M$ is equipped with a holomorphic Cartan geometry
$(E_H\, ,\theta)$ of type $G/H$.
\end{itemize}
Then there is an \'etale Galois holomorphic covering map
$$
\gamma\, :\, A\, \longrightarrow\, M\, ,
$$
where $A$ is a complex torus.
\end{theorem}

\begin{proof}
Consider the principal $G$--bundle $E_G$ in \eqref{eg}.
Recall that $\theta$ defines
a holomorphic connection on $E_G$. A holomorphic connection on $E_G$
induces a holomorphic connection on the associated vector
bundle $\text{ad}(E_G)$.

Since $c_1(M)\,=\, 0$,
a theorem due to Yau says that the holomorphic tangent
bundle $TM$ admits a K\"ahler--Einstein
metric \cite{Ya}. Fix a K\"ahler--Einstein K\"ahler form
$\omega$ on $M$. The vector bundle $TM$ is polystable because
$\omega$ is K\"ahler--Einstein; see \cite[p. 177, Theorem (8.3)]{Kob}
Hence from Lemma \ref{lem1},
\begin{equation}\label{eq0}
(c_2(TM)\cup \omega^{d-2})\cap [M] \,\geq\, 0\, .
\end{equation}

Since $\text{ad}(E_G)$ admits
a holomorphic connection, and $TM$ is polystable with
$$
\text{degree}(TM)\, =\, 0\, ,
$$
we conclude that the vector bundle $\text{ad}(E_G)$
is semistable (see \cite[p. 2830]{Bi1}). Also, we have
\begin{equation}\label{eq2}
c_j(\text{ad}(E_G))\, =\, 0
\end{equation}
for all $j\, \geq\, 1$ because the vector
bundle $\text{ad}(E_G)$ admits
a holomorphic connection \cite[p. 192--193, Theorem 4]{At}.

Consider the short exact sequence of vector bundles in \eqref{eq1}.
Since $c_1(M)\,=\,
0$, from \eqref{eq1} and \eqref{eq2} it follows that
\begin{equation}\label{eq3}
c_1(\text{ad}(E_H))\, =\, 0\, .
\end{equation}
Since $\text{ad}(E_H)$ is a subbundle of the semistable vector
bundle $\text{ad}(E_G)$ of degree zero, we conclude that
$\text{ad}(E_H)$ is also semistable (any subsheaf of
$\text{ad}(E_H)$ violating the semistability condition of
$\text{ad}(E_H)$ would also violate the semistability condition of
$\text{ad}(E_G)$). Now from Lemma \ref{lem1}
and \eqref{eq3} we conclude that
\begin{equation}\label{eq4}
(c_2(\text{ad}(E_H))\cup \omega^{d-2})\cap [M] \,\geq\, 0\, .
\end{equation}

In view of \eqref{eq2}, \eqref{eq3} and \eqref{eq4}, from
the short exact sequence in \eqref{eq1} it follows that
$$
(c_2(TM)\cup \omega^{d-2})\cap [M]\,=\,
-(c_2(\text{ad}(E_H))\cup\omega^{d-2})\cap [M]
\,\leq\, 0\, .
$$
On the other hand, from Lemma \ref{lem1},
$$
(c_2(TM)\cup\omega^{d-2})\cap [M]\,\geq\, 0\, .
$$
Hence
$$
(c_2(TM)\cup\omega^{d-2})\cap [M]\,=\, 0\, .
$$

Consequently,
$$
((2d\cdot c_2(TM) - (d-1)c_1(TM)^2)\cup\omega^{d-2})\cap [M]
\,=\, 0\, .
$$
Therefore, from Corollary 3 in \cite[p. 40]{BS} we conclude that
$TM$ is projectively flat. In particular,
$$
End(TM)\, =\, TM\otimes\Omega^1_M
$$
is a flat vector bundle. Hence $c_2(End(TM))\,=\, 0$. But
$$
c_2(End(TM))\,=\, 2d\cdot c_2(TM)-(d-1)c_1(TM)^2\, .
$$
Since $c_1(TM)\,=\, 0$, we conclude that $c_2(TM)\,=\, 0$.
Now from \cite[p. 248, Corollary 2.2]{IKO} (see also Lemma 1 of
\cite{Mc3}) we know that $M$ is holomorphically
covered by a complex torus. This completes the proof of
the theorem.
\end{proof}

\begin{remark}\label{rem1}
{\rm The holomorphic tangent bundle $TM$ of a compact connected
K\"ahler manifold $M$ admits a holomorphic connection
if and only if $M$ admits an \'etale covering by a complex torus.
Indeed, if $A\, \longrightarrow\, M$ is a covering map, where $A$ is
a complex torus, then the flat connection on $TA$ given by a 
trivialization of $TA$ descends to a holomorphic flat connection
on $TM$. To prove the converse, assume that $TM$ admits a
holomorphic connection. So $c_1(TM)\,=\, 0\,=\,
c_2(TM)$ \cite[p. 192--193, Theorem 4]{At}. Since
$TM$ is polystable \cite{Ya}, we know that $TM$ is
projectively flat \cite[p. 40, Corollary 3]{BS}. Since $c_1(TM)
\,=\, 0$, and $TM$ is projectively flat, it follows that
a K\"ahler--Einstein metric on $M$ is flat. Therefore,
$M$ is covered by a complex torus.}

{\rm It should be pointed out that even when $TM$ admits a
holomorphic connection, the exact sequence in \eqref{eq1} is not
compatible with the holomorphic connection on $\text{ad}(E_G)$
(unless $\dim M \,=\, 0$)
in the sense that the connection on $TM$ is not
a quotient of the connection on $\text{ad}(E_G)$.}
\end{remark}

\section{Cartan geometries and rational curves}\label{sec4}

A smooth complex projective variety $Z$ is
said to be \text{rationally
connected} if for any two given points of $Z$, there exists an 
irreducible rational curve on $Z$ that contains them; see
\cite[p. 433, Theorem 2.1]{KMM} and
\cite[p. 434, Definition--Remark 2.2]{KMM} for other equivalent
conditions. Note that a rationally connected
projective manifold is connected.

As before, $G$ is a complex Lie group and $H\,\subset\, G$
a closed complex subgroup. The quotient map $G\,\longrightarrow
\, G/H$ defines a holomorphic principal $H$--bundle over $G/H$.
There is a canonical holomorphic Cartan connection $\theta_0$
on this $H$--bundle $G\,\longrightarrow
\, G/H$ given by the identification of $\mathfrak g$
with the left--invariant vector fields on $G$.
We will call $(G/H\, ,G\, ,\theta_0)$ the \textit{tautological
Cartan geometry}.

\begin{theorem}\label{thm2}
Let $(M\, ,E_H\, ,\theta)$ be a holomorphic Cartan geometry
of type $G/H$ such that $M$ is rationally connected. Then
the homogeneous space $G/H$ is a rational projective variety, and
$M\, \cong\, G/H$. Furthermore,
$(M\, ,E_H\, ,\theta)$ is holomorphically isomorphic to
the tautological Cartan geometry $(G/H\, , G\, ,\theta_0)$.
\end{theorem}

\begin{proof}
Consider the principal $G$--bundle $E_G$ defined in
\eqref{eg}. Let $D$ denote the holomorphic connection on
the principal $G$--bundle $E_G$ induced by $\theta$ (see
Section \ref{sec3}). Since
$M$ is rationally connected, the curvature of
$D$ vanishes (see \cite[p. 160,
Theorem 3.1]{Bi2}). From the fact that $M$ is rationally 
connected it also follows that $M$ is
simply connected \cite[p. 545, Theorem 3.5]{Ca},
\cite[p. 362, Proposition 2.3]{Kol}.

Since $D$ is flat, and $M$ is
simply connected, we have
the developing map for the Cartan geometry
$(M\, ,E_H\, ,\theta)$
$$
\gamma\, :\, M\, \longrightarrow\, G/H
$$
which is a submersion \cite{S},
\cite[p. 3, Theorem 2.7]{Mc1}. Hence $\gamma$ is a covering
map because $M$ is compact. Since $G/H$ is covered by the
rationally connected complex projective manifold $M$, we
conclude that $G/H$ is also a rationally connected complex
projective manifold. We noted above that this implies
that $G/H$ is simply connected \cite[p. 545, Theorem 3.5]{Ca},
\cite[p. 362, Proposition 2.3]{Kol}.
Since $\gamma$ is a covering map, and $G/H$ is simply connected,
it follows that $\gamma$ is an isomorphism.

Since $G/H$ is a simply connected homogeneous projective variety,
we know that $G/H$ is rational. This completes the proof of the
theorem.
\end{proof}


\end{document}